\documentclass[11pt]{amsart}
\usepackage[top=1in, bottom=1in, left=1.25in, right=1.25in]{geometry}
\usepackage{graphicx}
\usepackage{mathrsfs}
\usepackage[mathscr]{eucal}
\usepackage{microtype}
\usepackage{hyperref}
\usepackage[parfill]{parskip}
\usepackage{amssymb}
\usepackage[dvipsnames]{xcolor}
\usepackage[percent]{overpic}
\usepackage{enumerate}
\usepackage{url}

\newcommand{\br}{\mathbb{R}}

\newcommand{\bz}{\mathbb Z}

\newcommand{\cc}{\mathcal C} 
\newcommand{\ca}{\mathcal A} 
\newcommand{\mcg}{\mathrm{MCG}} 
\newcommand{\pmcg}{\mathrm{PMCG}} 
\newcommand{\orb}{\mathcal{O}} 
\newcommand{\cg}{\mathcal G} 
\newcommand{\cp}{\mathcal P}
\newcommand{\al}{\alpha}
\newcommand{\be}{\beta}

\newcommand{\co}{\colon\thinspace}
\newcommand{\cv}{\mathcal{V}}

\DeclareMathOperator{\diam}{diam}
\DeclareMathOperator{\Ends}{Ends}
\DeclareMathOperator{\genus}{genus}
\DeclareMathOperator{\sep}{Sep}
\DeclareMathOperator{\nonsep}{NonSep}

\DeclareMathOperator{\fii}{\mathfrak{f}}

\newtheorem{theorintro}{Theorem}
\newtheorem{Thm}{Theorem}[section]
\newtheorem{Prop}[Thm]{Proposition}
\newtheorem{Lem}[Thm]{Lemma}
\newtheorem{Cor}[Thm]{Corollary}
\newtheorem{Conj}[Thm]{Conjecture}
\newtheorem{Question}[Thm]{Question}

\newtheorem*{mquestion*}{Motivating Question}
\newtheorem*{mthm}{Main Theorem}

\theoremstyle{definition}
\newtheorem{Def}[Thm]{Definition}

\theoremstyle{remark}
\newtheorem{Rem}[Thm]{Remark}

\numberwithin{equation}{section}

\title[Graphs of curves on infinite-type surfaces with mapping class group actions]{Graphs of curves on infinite-type surfaces\\ with mapping class group actions}

\author{Matthew Gentry Durham}
\address[Matthew Gentry Durham]{Department of Mathematics, Yale University}
\email{matthew.durham@yale.edu}

\author{Federica Fanoni}
\address[Federica Fanoni]{Mathematics Institute, University of Heidelberg}
\email{federica.fanoni@gmail.com}

\author{Nicholas G.\ Vlamis}
\address[Nicholas G.\ Vlamis]{Department of Mathematics, University of Michigan}
\email{vlamis@umich.edu}

\begin{document}

\begin{abstract}
We study when the mapping class group of an infinite-type surface $S$ admits an action with unbounded orbits on a connected graph whose vertices are simple closed curves on $S$. We introduce a topological invariant for infinite-type surfaces that determines in many cases whether there is such an action.  This allows us to conclude that, as non-locally compact topological groups, many big mapping class groups have nontrivial coarse geometry in the sense of Rosendal.
\end{abstract}
\maketitle

\vspace{-20pt}

\section{Introduction}

This article is concerned with the existence of robust generalizations of the standard curve graph to the context of infinite-type surfaces, i.e. surfaces with infinitely-generated fundamental groups.
The main interest is to build geometrically meaningful actions of mapping class groups on metric spaces constructed from  topological data of surfaces.
All surfaces considered are connected, orientable, and without boundary.

In the finite-type setting, the curve graph is known to be connected, infinite diameter, and Gromov hyperbolic with pseudo-Anosov mapping classes acting with positive translation length \cite{MasurHyperbolicityI}.  However, the curve graph of an infinite-type surface always has diameter two and as such it is trivial from the perspective of coarse geometry. 

Many authors have recently investigated numerous analogues of curve, arc, and pants graphs for specific classes of infinite-type surfaces \cite{BavardHyperbolic, AramayonaArc, FossasParlier,  AramayonaGeometry} (see \S \ref{intro:context}). 
The cases where the above constructions yield geometrically interesting actions of the mapping class group  require the underlying surface to have a finite positive number of isolated planar ends or positive finite genus, respectively.
The condition on the ends is required to study properly embedded arcs in the surface.
In order to work with infinite-type surfaces in general---removing these restrictions on the topology---we focus on curves in the surface, which leads us to ask:

\begin{mquestion*}
When does the mapping class group $\mcg(S)$ of an infinite-type surface act with unbounded orbits on a connected graph consisting of curves?
\end{mquestion*}

By a graph \textit{consisting of curves}, we mean a graph whose vertices correspond to isotopy classes of essential simple closed curves.
In order to answer this question (for but a few classes of surfaces), we introduce the \emph{finite-invariance index}, denoted $\mathfrak{f}(S)$, which associates to a surface $S$ an element of $\mathbb{N}\cup \{0, \infty\}$ (see Definition \ref{def:fii}).
Our results can be summarized as follows:

\begin{mthm} \label{thm:meta}
If $\fii(S)\geq 4$, then $\mcg(S)$ admits an unbounded action on a graph consisting of curves.  
If $\fii(S) = 0$, then $\mcg(S)$ admits no such unbounded action. 
\end{mthm}

When $\fii(S) \geq 4$, we construct an explicit graph on which $\mcg(S)$ acts with unbounded orbits.
In addition, we provide a sufficient condition for these graphs to be (uniformly) Gromov hyperbolic (see Theorem \ref{thm:main}).  
When $\fii(S) \in \{2,3\}$, we prove that $\fii(S)$ is too coarse an invariant to determine whether such an unbounded action exists; we give specific examples in \S \ref{sec:oddities}.  In the case of $\fii(S)=1$, we conjecture that $\mcg(S)$ admits no unbounded action on a graph consisting of curves.  

\begin{Rem}
The requirement of an infinite-diameter mapping class group orbit rules out examples of graphs with little topological significance.  For instance, if $S$ is a surface with countably many orbits of curves $\{\orb_i\}_{i\in\mathbb{Z}}$, we can consider the following graph: the vertices are all curves and there is an edge between $a\in \orb_i$ and $b\in\orb_j$ if and only if $|i-j|\leq 1$. This graph is quasi-isometric to $\bz$, so it has infinite-diameter (and it is Gromov hyperbolic), but each orbit of the mapping class group action has diameter two.
\end{Rem}

\begin{Rem}
A stronger version of the Main Theorem holds for the pure mapping class $\pmcg(S)$, the subgroup of $\mcg(S)$ which fixes $\Ends(S)$; see Corollary \ref{cor:pure} below.
\end{Rem}

\subsection{Motivation and context} \label{intro:context}
The \emph{mapping class group} of an oriented surface $S$, denoted $\mcg(S)$, is the group of orientation preserving homeomorphisms of $S$ modulo homotopy.
The algebraic and geometric structure of $\mcg(S)$ for finite-type surfaces is largely well understood; but, to date, \emph{big} mapping class groups, i.e. mapping class groups of infinite-type surfaces, remain mysterious.
There has been recent interest in better understanding the structure of these groups as one is led to them quite naturally; for instance, from studying group actions on finite-type surfaces (see \cite{CalegariCircular, CalegariBig}), constructing foliations of 3-manifolds (see \cite{CantwellEndperiodic}), and from the Artinization of automorphism groups of trees and stable properties of mapping class groups (see \cite{FunarInfinite, FunarBraided, FunarBraided2, FunarUniversal}).
For a discussion of these connections, see \cite{CalegariBig2}.

This article is largely motivated by the recent work of Bavard \cite{BavardHyperbolic}, where she constructed an action of $\mcg(\br^2 \smallsetminus C)$ on an infinite-diameter Gromov hyperbolic graph, where $C\subset \br^2$ is a Cantor set.
Furthermore, she used this action to prove that the bounded second cohomology of $\mcg(\br^2\smallsetminus C)$ is infinite dimensional.
These results answered a series of questions posed by Calegari \cite{CalegariBig}.  

By isolating the key properties behind the  graph constructed in Bavard's work, Aramayona--Fossas--Parlier \cite{AramayonaArc} extended Bavard's construction to a larger family of surfaces while preserving the geometric properties.  This was followed by Aramayona-Valdez \cite{AramayonaGeometry} in which the authors gave necessary and sufficient conditions for this construction to hold. 

This recent work on big mapping class groups has been focused on constructing actions on Gromov hyperbolic spaces; the hope is to mimic the theory for finite-type surfaces.
The curve graph of a surface $S$ is the graph, denoted $\cc(S)$, whose vertices correspond to isotopy classes of essential simple closed curves where adjacency denotes the existence of disjoint representatives. 
In the finite-type setting, the curve graph proved extremely useful in studying mapping class groups, for instance in understanding  their coarse geometry (e.g.\ to show that the mapping class group has finite geometric rank \cite{BehrstockDimension} and finite asymptotic dimension \cite{BBF10} and to prove quasi-isometric rigidity \cite{BehrstockGeometry, BowditchLarge}) or to study their cohomological properties \cite{HarerStability, HarerVirtual}.

A striking difference between the finite- and infinite-type setting is that the mapping class group is finitely generated in the former case but not in the latter, since big mapping class groups are uncountable.  This seems discouraging from the viewpoint of geometric group theory; however, there has been recent progress (see  \cite{RosendalCoarse}) in studying the coarse geometry of topological groups that are neither finitely generated nor locally compact.
A particularly beautiful application of this theory in \cite{MannLarge} is the study of the large-scale geometry of homeomorphism groups of compact manifolds and its relation to the topology of and dynamics on the manifold.

By equipping $\mathrm{Homeo^+(S)}$ with the compact-open topology, the mapping class group is 
\[
\mcg(S) = \pi_0(\mathrm{Homeo}^+(S)).
\]
Giving $\mcg(S)$ the associated quotient topology, it becomes a topological group.
With this topology, it is an easy exercise to see that if $\mcg(S)$ acts on a graph consisting of curves, then it does so continuously.
The actions with unbounded orbits constructed in the Main Theorem, when $\fii(S) \geq 4$, yield a left-invariant continuous infinite-diameter pseudo-metric on $\mcg(S)$:
Indeed, let $\Gamma$ be the graph guaranteed from the Main Theorem and fix a vertex $x \in \Gamma$.
The function $d\co \mcg(S) \times \mcg(S) \to \br$ defined by 
\[
d(f,g) = d_\Gamma(f(x), g(x))
\]
where $f,g \in \mcg(S)$ and $d_\Gamma$ is the graph metric on $\Gamma$ is the desired pseudo-metric.
By the Main Theorem, $(\mcg(S), d)$ has infinite-diameter.
In the language of \cite{RosendalCoarse}, as an immediate consequence, we have the following corollary, which informally says that $\mcg(S)$ is not coarsely a point; hence, $\mcg(S)$ is potentially amenable to investigation via these coarse geometric tools:

\begin{Cor}
If $\fii(S) \geq 4$, then $\mcg(S)$ does not have property (OB).
\end{Cor}

\begin{Rem}
There are no well-understood generating sets for $\mcg(S)$, so the existence of such a pseudo-metric is unclear \emph{a priori}.
\end{Rem}
\begin{Rem}
As a consequence of Corollary \ref{cor:pure} below, this statement holds for the pure mapping class group of any surface $S$ with at least four ends.
\end{Rem}

\subsection{Discussion of results}
\label{sec:discussion}

Let $S$ be any infinite-type surface, $\Ends(S)$ its space of ends, and $\mcg(S)$ its mapping class group.
We first define $\fii$ from the Main Theorem:

\begin{Def}
\label{def:fii}
We say that a collection $\cp$ of disjoint subsets of the space of ends is \emph{$\mcg(S)$-invariant} if for every $P\in\cp$ and for every $\varphi\in\mcg(S)$ there exists $Q\in\cp$ such that $\varphi(P)=Q$. The \emph{finite-invariance index} of $S$, denoted $\fii(S)$, is defined as follows:

\begin{itemize}
\item $\fii(S)\geq n$ if there is a $\mcg(S)$-invariant collection $\mathcal{P}$ of disjoint closed proper subsets of $\Ends(S)$ satisfying $|\mathcal{P}| = n$;
\item $\fii(S)=\infty$ if $\genus(S)$ is finite and positive; 
\item $\fii(S)=0$ otherwise. 
\end{itemize}
We say that $\fii(S)=n$ if $\fii(S)\geq n$ but $\fii(S)\ngeq n+1$.
\end{Def}

We first look at $\fii(S)=0$:

\begin{theorintro}\label{thm:zero}
If $\fii(S) = 0$, then every action of $\mcg(S)$ on a graph consisting of curves has finite-diameter orbits.
\end{theorintro}

As shown in the appendix (see Proposition \ref{prop:appendix}), if $\fii(S) = 0$, then $S$ is either
\begin{enumerate}
\item
the \emph{Cantor tree surface} (i.e.\ the sphere minus a Cantor set),
\item
the \emph{blooming Cantor tree surface} (i.e.\ the infinite-genus surface with no planar ends in which $\Ends(S)$ is a Cantor set), or
\item
the \emph{Loch Ness monster surface} (i.e.\ the infinite-genus surface with a single end).
\end{enumerate}

Let us turn to $\fii(S)\geq 4$, where we give an explicit construction of the graph in the Main Theorem.
Let $\cp$ be a finite collection of pairwise disjoint closed subsets of $\Ends(S)$.
Define $\sep_2(S,\cp)$ to be the graph consisting of curves in which a curve $c$ is a vertex if it is separating and:
\begin{enumerate}[(i)]
\item the set of ends of each component of $S\setminus c$ contains at least two elements of $\cp$ and
\item every element of $\cp$ is contained in the set of ends of a component of $S\setminus c$.
\end{enumerate}
Loosely speaking, $c$ is a vertex if it partitions the elements of $\cp$ into sets of cardinality at least two and does not split the elements of $\cp$.
Adjacency denotes disjoint (except if $|\cp| = 4$, when adjacency denotes intersection number at most $2$).
Notice that $\sep_2(S, \cp)$ is an induced subgraph of the curve graph of $S$ (if $|\cp|\neq 4$).

\begin{theorintro}\label{thm:main}
If $|\cp|\geq 4$, $\sep_2(S,\cp)$ is connected and infinite diameter. Furthermore, there exist infinitely many elements in $\mcg(S)$ acting on $\sep_2(S,\cp)$ with positive translation length.
Moreover, if each element of $\cp$ is a singleton, then $\sep_2(S,\cp)$ is $\delta$-hyperbolic, where $\delta$ can be chosen independent of $S$ and $P$.
\end{theorintro}

Note that a bound for $\delta$ can be computed from the proof, but the estimates are presumably far from optimal.
Theorem \ref{thm:main} also holds for finite-type surfaces (see Theorem \ref{thm:sephyp}).
The pure mapping class group, denoted $\pmcg(S)$, is the subgroup of $\mcg(S)$ acting trivially on $\Ends(S)$.
Observe that restricting to $\pmcg(S)$, we have the following immediate corollary:

\begin{Cor}
\label{cor:pure}
Suppose $|\Ends(S)| \geq 4$.
Let $\cp$ be a finite collection of singletons of $\Ends(S)$ with $|\cp|\geq 4$, then $\sep_2(S,\cp)$ is connected, $\delta$-hyperbolic, and infinite diameter.
Furthermore, $\pmcg(S)$ acts on $\sep_2(S,\cp)$ and infinitely-many elements of $\pmcg(S)$ act with positive translation length.
\end{Cor}

Finally, the case where $\fii(S)=\infty$  and the surface has finite positive genus is handled by looking at the graph $\nonsep(S)$ of nonseparating curves (i.e.\ the induced subgraph of the curve graph on the set of nonseparating curves). This graph is studied by Aramayona--Valdez \cite{AramayonaGeometry}, where they prove:

\begin{theorintro}[{\cite[Theorem 1.4]{AramayonaGeometry}}]\label{thm:nonsep}
If $\genus(S)$ is finite and nonzero, then $\nonsep(S)$ is connected and  has infinite diameter; further, there exist infinitely many elements of $\mcg(S)$ acting with positive translation length.
\end{theorintro}

Recently Rasmussen \cite{RasmussenUniform} proved that $\nonsep(S_{g,p})$ is $\delta$-hyperbolic for $\delta>0$ independent of $g$ and $p$, provided $g \geq 2$.
Combining this with a result of Aramayona-- Valdez \cite{AramayonaGeometry}, it follows that $\nonsep(S)$ is Gromov hyperbolic if $\genus(S)$ is at least two and finite.

\subsection{Outline}

In \S \ref{sec:conventions}, we give the basic definitions we will need and in \S \ref{sec:background}, we include the necessary background on the classification of infinite-type surfaces and the structure of the space of ends.
We prove Theorem \ref{thm:zero}  in \S \ref{sec:finitediam}.
The idea behind the proof is that there are mapping classes sending any simple closed curve arbitrarily far out into an end of the surface.

In \S \ref{sec:mainthmfinite}, we prove Theorem \ref{thm:main} for finite-type surfaces.
The proof goes through studying a subgraph of the arc graph.
Using the tools of \cite{HenselSlim}, we are able to ``guess geodesics" in our subgraph that form uniformly slim triangles.
A lemma from \cite{BowditchUniform} then yields the uniform hyperbolicity of these subgraphs.
We will see that these subgraphs are uniformly quasi-isometric to the graphs of interest allowing us to import the uniform hyperbolicity.

In \S \ref{sec:quasiretracts}, we abstract the argument in \cite{AramayonaArc} to the setting of an arbitrary geodesic metric space.
We use this framework to promote the proof of Theorem \ref{thm:main} for finite-type surfaces to the infinite-type setting in \S \ref{sec:mainthminfinite}.
Building off the philosophy in \cite{BavardHyperbolic, AramayonaArc}, we will see that the finiteness of the collection $\cp$ forces our curves to interact with a compact region of our surface yielding the infinite-diameter property.
In the case $\cp$ consists of singletons, we use the fact that any finite collection of curves can only fill a compact surface; we will then apply Theorem \ref{thm:main} to this compact surface to see that triangles are uniformly slim.

In \S \ref{sec:nonsep}, we give a quick proof of Theorem \ref{thm:nonsep}.
In \S \ref{sec:oddities}, we explore some of the oddities of the low-index cases by constructing examples both of surfaces that do admit a graph with an interesting action of the mapping class group and of surfaces that do not.

Finally, in the appendix we give a more detailed description of the space of ends. Even though most of this material does not show up in our proofs, it was fundamental to motivate the constructions.
We include it in hopes that it will aid other researchers thinking about these surfaces.

\subsection*{Acknowledgements}
The second author is grateful to Brian Bowditch for many useful discussions and to Bram Petri for suggesting part of the proof of Lemma \ref{lem:Gconnected}.
The third author thanks Javier Aramayona for helpful conversations.
The authors would also like to thank Javier Aramayona and Ferr\'an Valdez for sharing the manuscript of \cite{AramayonaGeometry}, Dan Margalit for a helpful conversation, and the referee for their careful reading and valuable comments.

The second author acknowledges support of Swiss National Science Foundation Grant Number P2FRP2\_161723 and from U.S. National Science Foundation grants DMS 1107452, 1107263, 1107367 ``RNMS: Geometric Structures and Representation Varieties'' (the GEAR Network). The first and third author were supported in part by NSF RTG grant 1045119.  The first author also gratefully acknowledges the support of MSRI.

\section{Conventions and standard definitions}\label{sec:conventions}

All surfaces appearing in our results have negative Euler characteristic and are orientable, connected, separable, and without boundary.
A surface is of \emph{finite (topological) type} if its fundamental group is finitely generated.
Otherwise, it is of \emph{infinite (topological) type}.
The \textit{mapping class group} of a surface $S$, denoted $\mcg(S)$, is the group of orientation preserving homeomorphisms modulo homotopy, or, as mentioned earlier, if we equip $\mathrm{Homeo}^+(S)$ with the compact-open topology, then $\mcg(S) = \pi_0(\mathrm{Homeo}^+(S))$.
The subgroup of $\mcg(S)$ acting trivially on $\Ends(S)$ is the \emph{pure mapping class group}, denoted $\pmcg(S)$.
  
A simple closed curve on a surface is \textit{peripheral} if it is isotopic to an end of $S$; it is \textit{essential} if it is neither peripheral nor bounds a disk; it is \textit{separating} if its complement is disconnected and \emph{nonseparating} otherwise.
A simple arc on a surface is \emph{proper} if the endpoints are contained in $\Ends(S)$; it is \emph{essential} if it is not isotopic to an end of $S$.
Two essential curves $a$ and $b$ \emph{fill} a subsurface $\Sigma$ if every essential curve in $\Sigma$ intersects at least one of $a$ or $b$.
In addition we require each boundary component of $\Sigma$ to be essential in $S$, so then $\Sigma$ is unique up to isotopy.
With this definition, $\Sigma$ may not be compact, but will always be finite-type.

The \emph{curve graph} of a surface $S$, denoted $\cc(S)$, is the graph whose vertices correspond to isotopy classes of essential simple closed curves and two vertices are adjacent if they have disjoint representatives.
Similarly, the \emph{arc graph} of a surface $S$, denoted $\ca(S)$, is the graph whose vertices correspond to isotopy classes of essential simple proper arcs and two vertices are adjacent if they have disjoint representatives.
We will view graphs as metric spaces in which the metric is defined by setting each edge to have length $1$.

We will regularly abuse notation and conflate a vertex in $\cc(S)$ or $\ca(S)$ with a representative on the surface.  
When discussing representatives of collections of vertices, we will always assume they are pairwise in minimal position (this can be done for instance by taking geodesic representatives in a complete hyperbolic metric).

In our proofs, we will need to consider compact surfaces with nonempty boundary and their associated curve and arc graphs.  Given such a surface, note that $\cc(S) = \cc(\mathrm{int}(S))$ and $\ca(S) = \ca(\mathrm{int}(S))$, where $\mathrm{int}(S)$ denotes the interior of $S$.  

\section{Background on infinite-type surfaces}\label{sec:background}
Infinite-type surfaces are classified by their genus and their space of ends.
This classification was originally due to Ker\'ekj\'art\'o; however, Richards in \cite{RichardsClassification} simplified and filled some gaps in the original proof.

We will now define an end and describe how to topologize the set of ends.
We say that a subset $A$ of a surface $S$ is \emph{bounded} if its closure in $S$ is compact.
\begin{Def}We call a descending chain $U_1 \supset U_2 \supset \cdots$ of connected unbounded open sets in $S$ \emph{admissible}  if it satisfies
\begin{enumerate}[(a)]
\item
the boundary of $U_n$ in $S$ is compact for all $n$ and
\item
for any compact set $K$ of $S$, $U_n \cap K = \emptyset$ for sufficiently large $n$.
\end{enumerate}\end{Def}
We say two admissible descending chains $U_1 \supset U_2 \supset \cdots$ and $V_1 \supset V_2\supset \cdots$ are equivalent if for any $n$ there exists an $N$ such that $U_N \subset V_n$ and vice versa.
\begin{Def}An \textit{end} of $S$ is an equivalence class of admissible descending chains.\end{Def}
As a set, the \textit{space of ends} of $S$, denoted $\Ends(S)$, consists of the ends of $S$.
$\Ends(S)$ is topologized as follows:  given an open set $U$ of $S$ with compact boundary, let $U^*\subset \Ends(S)$ be the set of points of the form $p =[ V_1 \supset V_2 \supset \cdots]$ with $V_n \subset U$ for $n$ sufficiently large.
We take all such $U^*$ as a basis for the topology on $\Ends(S)$.
Proposition \ref{prop:ends} below describes the topology of $\Ends(S)$; a proof can be found in \cite[Chapter 1, \S36 and 37]{AhlforsRiemann}.
The structure of $\Ends(S)$ is expounded on in more detail in the appendix.

\begin{Prop}
\label{prop:ends}
The space of ends of a surface is totally disconnected, separable, and compact.
\end{Prop}

\begin{Def}
An end $[U_1 \supset U_2 \supset \cdots]$ of $S$ is 
\begin{itemize}

\item
\textit{planar}, or a \textit{puncture}, if $U_n$ is planar (i.e.\ can be embedded in the plane) for $n$ sufficiently large;

\item
\textit{isolated} if the singleton $\{[U_1\supset U_2\supset\cdots]\}$ is open in $\Ends(S)$;

\item
\textit{accumulated by genus} if $U_n$ has infinite genus for all $n$.

\end{itemize}
\end{Def}
Let $\mathscr{AG}(S) \subset \Ends(S)$ denote the set of ends accumulated by genus. We can then associate to an orientable surface $S$ the triple $(\genus(S), \Ends(S), \mathscr{AG}(S))$.
It is a theorem of Ker\'ekj\'art\'o (see \cite[Theorem 1]{RichardsClassification}) that this triple uniquely determines $S$ up to homeomorphism.

For the sake of discussing the homeomorphism type of a simple closed curve, we note:
If $S$ has finitely many boundary components, then a similar classification holds by also recording the number of boundary components.
This is also covered by \cite[Theorem 1]{RichardsClassification}.

\section{Graphs with finite-diameter mapping class group orbits}\label{sec:finitediam}
In this section we show that the mapping class group of a surface with $\fii=0$ does not admit interesting geometric actions on graphs consisting of curves. To prove this result, we will use the following general criterion:
\begin{Prop}\label{prop:criterionfd}
Let $S$ be an oriented surface and $\Gamma = \Gamma(S)$ be a connected graph consisting of curves on which the mapping class group acts. Let $\cv \subset \Gamma \times \Gamma$ satisfying:
\begin{enumerate}
\item there exists a vertex $c\in \Gamma$ such that, up to the action of $\mcg(S)$, there is a finite number of pairs $(a,b)\in\cv$ with $a,b \in \mcg(S)\cdot c$, and
\item for every $a,b\in \mcg(S)\cdot c$ with $(a,b)\notin\cv$, there exists $d \in \mcg(S)\cdot c$ such that $(a,d)$ and $(b,d)$ belong to $\cv$.
\end{enumerate}  Then every $\mcg(S)$-orbit in $\Gamma$ has finite diameter.
\end{Prop}
\begin{proof}
We first show that the mapping class group orbit of $c$, denoted $\orb_c$, has finite diameter. By condition (1) and the mapping class group invariance of $\Gamma(S)$, 
$$A:=\sup\{d(a,b)\, |\, a,b\in\orb_c \mbox{ and } (a,b)\in \cv\}<\infty.$$
Consider any pair $a,b\in \orb_c$. If $(a,b)\in \cv$, $d(a,b)\leq A$. Otherwise, by condition (2) there is a curve $d \in \orb_c$ such that $(a,d)$ and $(b,d)$ belong to $\cv$. 
So
$$d(a,b)\leq d(a,d)+d(d,b)\leq 2A,$$
which means that $\diam(\orb_c)\leq 2A$

Consider now another orbit $\orb$. Fix a curve $a\in\orb$ and set $B:=d(a,c)$. Then for any other curve $b\in\orb$ there exists a mapping class sending $a$ to $b$ and $c$ to some $d\in\orb_c$. Again, by the mapping class group invariance, $d(b,d)=d(a,c)=B$, which means that the distance of any $b\in \orb$ to $\orb_c$ is at most $B$. As the diameter of $\orb_c$ is at most $2A$, any two curves in $\orb$ are at distance at most $2B+2A$.
\end{proof}

Applying Proposition \ref{prop:criterionfd}, we obtain the following result, which we will use to prove Theorem \ref{thm:zero}. At the same time, Proposition \ref{prop:conditions} also shows that there are general restrictions for a graph to be connected and have infinite-diameter mapping class group orbits.

\begin{Prop}\label{prop:conditions}
Suppose that $\Gamma=\Gamma(S)$ is a connected graph consisting of curves with an action of $\mcg(S)$ and:
\begin{enumerate}
\item
$S$ has infinitely many isolated punctures and $\Gamma$ contains a vertex bounding a finite-type genus-0 surface, or
\item $\genus(S)=\infty$, $S$ has either no punctures or infinitely many isolated punctures, and $\Gamma$ contains a curve bounding a finite-type surface, or
\item $\genus(S)=\infty$ and $\Gamma$ contains a nonseparating curve,
\end{enumerate}
then $\mcg(S)$ acts on $\Gamma(S)$ with finite-diameter orbits.
\end{Prop}

\begin{proof}
We want to check that in all cases the hypotheses of Proposition \ref{prop:criterionfd} are satified.

For cases (1) and (2), we set $\cv=\{(a,b)\,|\,i(a,b)=0\}$.

\emph{Case (1)} Let $c$ be a vertex bounding a disk with $n$ punctures. So the mapping class group orbit of $c$ is given by all curves bounding $n$ punctures. There is a unique orbit of pairs of disjoint curves in $\mcg(S)\cdot c$ as the complementary regions of any two are two $n$-punctured disks and a surface homeomorphic to $S$. Moreover, if two curves $a$ and $b$ in the orbit of $c$ intersect, there is a complementary component which contains an infinite number of isolated punctures. So we can choose a curve in that component bounding exactly $n$ punctures and it will be disjoint from both $a$ and $b$.

\emph{Case (2)} This is analogous to case (1), with the  genus together with the isolated punctures playing the role of the isolated punctures in (1). 

\emph{Case (3)} Let $c$ be any nonseparating curve and set $$\cv=\{(a,b)\,|\, i(a,b)=0 \mbox{  and } a\cup b \mbox{  does not separate}\}.$$ Note that in particular any curve in a pair belonging to $\cv$ must be nonseparating and that there is a unique $\mcg(S)$-orbit of nonseparating pairs of (nonseparating) curves. Furthermore, if a pair  $(a,b)$ of nonseparating is not an element of $\cv$, then there is an infinite-genus component of the complement of $a \cup b$.
In this component there is a nonseparating curve $d$ such that $(a,d) \in \cv$ and $(b,d) \in \cv$.
\end{proof}

We are now ready to prove Theorem \ref{thm:zero}.

\begin{proof}[Proof of Theorem \ref{thm:zero}]
Let $\Gamma=\Gamma(S)$ be a connected graph consisting of curves on which $\mcg(S)$ acts.

If $S$ is the Cantor tree surface, we apply Proposition \ref{prop:criterionfd}, with $\cv=\{(a,b)\,|\,i(a,b)=0\}$. All curves are in the same mapping class group orbit and there is a unique orbit of pairs of disjoint curves. Indeed, the complement of a pair of disjoint curves is the union of two disks and an annulus each with a Cantor set removed. Moreover, if two curves intersect, each complementary component of their union contains a Cantor set. So we can choose any essential curve in the complement to verify the second condition of Proposition \ref{prop:criterionfd}.

Suppose $S$ is the blooming Cantor tree surface. If $\Gamma$ contains a nonseparating curve or a curve bounding a finite-type surface, then it has finite-diameter orbits by Proposition \ref{prop:conditions}.
Otherwise, $\Gamma$ must contain a separating curve $c$ with infinite-type complementary components. In this case, we apply Proposition \ref{prop:criterionfd} with 
\[\cv=\{(a,b)\,|\,\mbox{$a$ and $b$ do not cobound a finite-type subsurface}\}.
\] 
There is a unique mapping class group orbit of pairs $(a,b)\in \cv$ with $a,b\in\mcg(S)\cdot c$: indeed, the complement of any such pair is given by three bordered blooming Cantor tree surfaces, two with a single boundary component and one with two boundary components. Moreover, it is easy to verify that the second condition of Proposition \ref{prop:criterionfd} holds.

Finally, if $S$ is the Loch Ness monster surface, each curve is either nonseparating or it bounds a finite-type surface, so by Proposition \ref{prop:conditions} the graph $\Gamma$ has finite-diameter orbits.
\end{proof}

\section{Theorem \ref{thm:main} for finite-type surfaces}\label{sec:mainthmfinite}
In this section, we will prove Theorem \ref{thm:main} for finite-type surfaces.  
The subsequent two sections are then dedicated to extending to the infinite-type case.  
For standard background on the coarse geometry in this section we refer the reader to \cite[\S I.8, \S II.6, and \S III.H.1]{BridsonMetric}.  

\begin{Def}Let $\cp$ is a collection of pairwise disjoint subsets of $\Ends(S)$; we define $\sep_2(S,\cp)$ to be the subgraph of $\cc(S)$ spanned by the following vertices: a curve $c$ is a vertex in $\sep_2(S,\cp)$ if it is separating and if
\begin{enumerate}[(i)]
\item
for each $P \in \cp$ there is a component $S'$ of $S\smallsetminus c$ such that $P \subset \Ends(S')$, and
\item
for each component $S'$ of $S\smallsetminus c$ there exist two distinct elements $P,Q\in\cp$ such that $P,Q\subset\Ends(S')$.
\end{enumerate}
\end{Def}
When $\cp$ is a collection of singletons, we will let $P$ denote the union of the singletons and write $\sep_2(S,P)$ for $\sep_2(S,\cp)$.

If $|\cp|=4$, instead of deeming two curves adjacent if they are disjoint, we say that they are adjacent if they intersect at most twice. For simplicity, we will assume throughout this section and \S \ref{sec:mainthminfinite} that $|\cp|\geq 5$, but the proofs can be readily adapted  for the case $|\cp|=4$.

We first focus on the case where $\cp$ is a collection of singletons.
The proof of Theorem \ref{thm:main} in this setting will go through proving the analogous result for a subgraph of the arc graph.
We will show that this subgraph is uniformly quasi-isometric to $\sep_2(S,\cp)$, which will yield the result.
We will end the section with removing the singleton assumption and giving an argument for connectedness using Putman's technique \cite[Lemma 2.1]{PutmanConnectivity}. 

\subsection{Arcs}
Let $F$ be a noncompact finite-type surface. The reader might find it helpful to think of punctures as marked points.
For a collection of punctures $Q$ of $S$, let $\ca_2(F,Q)$ be the induced subgraph of $\ca(F)$ on the arcs with distinct endpoints in $Q$. We will prove:

\begin{Thm}
\label{thm:arcs}
If $|Q| \geq 3$, then $\ca_2(F,Q)$ is connected and $\delta$-hyperbolic, where $\delta$ is independent of $F$ and $Q$.  Furthermore, there exist infinitely many elements of $\mcg(F)$ that act on $\ca_2(F,Q)$ with positive translation length.
\end{Thm}

For the proof will rely on the ``guessing geodesics lemma" \cite[Proposition 3.1]{BowditchUniform} (similar results appeared earlier in \cite{GilmanDefinition} and \cite{MasurDisk}) and the unicorn paths conctruction of \cite{HenselSlim}.  

\begin{Lem}[Guessing geodesics lemma]
\label{lem:bowditch}
Suppose $\Gamma$ is a graph where all edges have length one.  
Then $\Gamma$ is Gromov hyperbolic if and only if there is a constant $M \geq 0$ and for any two vertices $x$ and $y$ in $\Gamma$ there is a connected subgraph $A(x,y) \subset \Gamma$ containing $x$ and $y$, with the following properties:
\begin{itemize}
\item
(Local) If $d_\Gamma(x,y) \leq 1$, then $A(x,y)$ has diameter at most $M$.
\item
(Slim triangles) For any three vertices $x,y,z$ in $\Gamma$ the subgraph $A(x,y)$ is contained in the $M$-neighborhood of $A(x,z) \cup A(z,y)$.
\end{itemize}
Further, the hyperbolicity constant only depends on $M$.
\end{Lem}

The following two definitions were introduced in \cite{HenselSlim}.
\begin{Def}[Unicorn arcs]
Let $a$ and $b$ be arcs in $F$ in minimal position representing vertices of $\ca(F)$.
Let $\al$ and $\be$ be endpoints of $a$ and $b$, respectively, and let $\pi \in a \cap b$.
Let $a'$ and $b'$ be subarcs of $a$ and $b$, respectively, such that $\partial a' = \{\al, \pi\}$ and $\partial b = \{\pi, \be\}$, then $a' \cup b'$ represents a vertex of $\ca(F)$.
The arc $a' \cup b'$ is a \emph{unicorn arc obtained from $a^\al, b^\be$}.
\end{Def}

Observe that $a'\cup b'$ is determined by $\pi$, so the unicorn arcs obtained from $a^\al, b^\be$ are in bijection with the elements of $a \cap b$.
Order the unicorns arcs by $a'\cup b' \leq a'' \cup b''$ if $a'' \subset a'$ and let $\{c_1, \ldots, c_n\}$ be the ordered set of unicorn arcs obtained from $a^\al, b^\be$.
\begin{Def}[Unicorn paths]
The path $$\mathcal{U}(a^\al, b^\be) = (c_0 = a, c_1, \ldots, c_n, c_{n+1}=b)$$ is the \emph{unicorn path between $a^\al, b^\be$}. 
\end{Def}
Let $a^\al$ and $b^\be$ be arcs with endpoints such that $a,b\in \ca_2(F,Q)$.
If $\al \neq \be$, then $\mathcal{U}(a^\al, b^\be)$ is a path in $\ca_2(F,Q)$; in particular, $\ca_2(F,Q)$ is connected.  
For $a,b \in \ca_2(F,Q)$, define
\[
A'(a,b) = \!\!\!\bigcup_{\substack{\al \in \partial a,\, \be\in \partial b\\ \al \neq \be}}\!\!\!\mathcal{U}(a^\al, b^\be).
\]
If $a$ and $b$ are adjacent, let $A(a,b)$ be the edge between them, otherwise let $A(a,b) = A'(a,b) \cup A'(b,a)$.
It is easy to see that the graphs $A(a,b)$ satisfy the local condition of Lemma \ref{lem:bowditch} for any $M\geq 1$.  It is left to verify the second condition. 

\begin{Lem}[{\cite[Lemma 3.3]{HenselSlim}}]
\label{lem:hensel}
Let $a,b,d$ be arcs mutually in minimal position.
For every $c\in \mathcal{U}(a^\al,b^\be)$ there is $c^*\in \mathcal{U}(a^\al, d^\delta)\cup \mathcal{U}(d^\delta, b^\beta)$ such that $c$ and $c^*$ are adjacent in $\ca(F)$.
\end{Lem}

\begin{Lem}[Slim triangles]
\label{lem:slim}
For all $a,b,d \in \ca_2(F,Q)$, the subgraph $A(a,b)$ is contained in the $2$-neighborhood of $A(a,d) \cup A(d,b)$.  
\end{Lem}

\begin{proof}
Unless $\partial a = \partial b = \partial d$, the result follows directly from Lemma \ref{lem:hensel}.  
So, assume that $\partial a = \partial b = \partial d = \{\al, \be\}$.  
Let $c\in \mathcal{U}(a^\al, b^\be) \subset A(a,b)$. Pick $d' \in \ca_2(F,Q)$ such that $d'$ is adjacent to $d$ and $\partial d' \neq \partial d$.
We can now pick $\delta \in \partial d'\smallsetminus \{\al,\be\}$ and apply Lemma \ref{lem:hensel} to $a^\al, b^\be, d'^\delta$ to get $c' \in \mathcal{U}(a^\al,d'^\delta)\cup \mathcal{U}(d'^\delta,b^\beta)$ such that $c$ and $c'$ are disjoint.
Assume that $c' \in \mathcal{U}(a^\al,d'^\delta)$.
Apply Lemma \ref{lem:hensel} to $a^\al, d'^\delta, d^\beta$ to get $c^* \in A(a,d)$ such that $d(c,c^*) \leq 2$.  
\end{proof}

\begin{proof}[Proof of Theorem \ref{thm:arcs}]
The existence of unicorn paths provides the connectedness. 
Lemma \ref{lem:slim} tells us the conditions of Lemma \ref{lem:bowditch} are met with $M = 2$ implying uniform hyperbolicity.
Finally, as $\ca_2(F,Q)$ is a subgraph of $\ca(F)$, any pseudo-Anosov element of the pure mapping class group acts with positive translations length \cite[Proposition 4.6]{MasurHyperbolicityI}.
\end{proof}

\subsection{Curves}
Moving to $\sep_2(F,Q)$, consider the natural map $\phi\co \ca_2(F,Q) \to \sep_2(F,Q)$ given by setting $\phi(a)$ to be the boundary of a regular neighborhood of $a$.

\begin{Lem}\label{lem:arcqi}
If $|Q|\geq 5$, then $\phi$ is a $(2,2)$-quasi-isometry.
\end{Lem} 

\begin{proof}
Let $d_\ca$ and $d$ denote distance in $\ca_2(F,Q)$ and $\sep_2(F,Q)$, respectively. 
Consider any $c\in\sep_2(F,Q)$.
If $a$ is any arc in a component of $F\smallsetminus c$ connecting two points of $Q$, then either $\phi(a)=c$ or $\phi(a)$ is disjoint from $c$. 
So, the image of $\phi$ is $1$-dense in $\sep_2(F,Q)$.

Observe that if $d_\ca(a,b) = 1$ with $\partial a \cap \partial b = \emptyset$, then $d(\phi(a),\phi(b)) = 1$. 
Further, if $d_\ca(a,b)=1$ with $\partial a \cap \partial b \neq \emptyset$, then there is an arc $c$ with endpoints in $Q\smallsetminus (\partial a \cup \partial b)$ disjoint from both $a$ and $b$.  
It follows that $d(\phi(a),\phi(b)) = 2$.  
We have shown that 
\[
d(\phi(a),\phi(b)) \leq 2d_\ca(a,b).
\]

We now focus on giving an inequality in the other direction.
Given $x \in \sep_2(F,Q)$, define $\Phi(x) = \{a \in \ca_2(F,Q) : a \text{ is disjoint from } x \}$.  We proceed by a sequence of claims:

\textit{Claim 1: If $\phi(a)$ is disjoint from $x \in \sep_2(F,Q)$, then $d_\ca(a,b) \leq 2$ for all $b \in \Phi(x)$.}

If $\phi(a)$ is disjoint from $x$, then so is $a$.
Therefore, if $b\in \Phi(x)$ with $d_\ca(a,b) > 1$, then $a,b$ are contained in the same component of $F\smallsetminus x$.
The other component must contain an element of $\ca_2(F,Q)$ implying $d_\ca(a,b) = 2$.

\textit{Claim 2: If $x,y \in \sep_2(F,Q)$ are disjoint and $a \in \Phi(x)$, then there exists $b \in \Phi(y)$ with $d_\ca(a,b) = 1$.}

As $F\smallsetminus (x \cup y)$ has three components, there must be a component disjoint from $a$ and containing at least two points of $Q$.
Let $b\in \ca_2(F,Q)$ be any arc in this component, then $b\in \Phi(y)$ and $d_\ca(a,b) =1$.

\textit{Claim 3: If $d(\phi(a),\phi(b)) \leq 1$, then $d_\ca(a,b)  \leq 1$.}

Observe that $\phi(a) = \phi(b)$ if and only if $a = b$.  
Now assume that $d(\phi(a),\phi(b))=1$.   
This can only happen if $a,b$ are disjoint and have distinct endpoints.  
In particular, $d_\ca(a,b) = 1$. 

We are now in a place to give the inequality.
Let $m = d(\phi(a), \phi(b))$, then, by Claim 3, we can assume that $m \geq 2$.
Let $\phi(a) = x_0, x_1, \ldots, x_m=\phi(b)$ be a path between $\phi(a)$ and $\phi(b)$.
Define $c_0 = a$ and choose $c_i \in \Phi(x_i)$ so that $d_\ca(c_{i-1},c_i) = 1$ as guaranteed by Claim 2.
By Claim 1, $d_\ca(c_m, b) \leq 2$.
This tells us that $d_\ca(a,b) \leq m+2$.
We have shown
\[
d_\ca(a,b) - 2 \leq d(\phi(a),\phi(b)) \leq 2 d_\ca(a,b)
\]
which yields the desired result.
\end{proof}

\begin{Thm}
\label{thm:sephyp}
Let $F$ be a finite-type surface and $Q$ a subset of $\Ends(F)$.
If $|Q|\geq 5$, then $\sep_2(F,Q)$ is connected, infinite-diameter, and $\delta$-hyperbolic, where $\delta$ can be chosen independent of $S$ and $Q$. 
Furthermore, there exist infinitely many elements of $\mcg(F)$ acting on $\sep_2(F,Q)$ with positive translation length.
\end{Thm}

\begin{proof}
The result follows immediately from Theorem \ref{thm:arcs} and Lemma \ref{lem:arcqi}.
\end{proof}

\subsection{Non-singletons}

\begin{figure}
\includegraphics{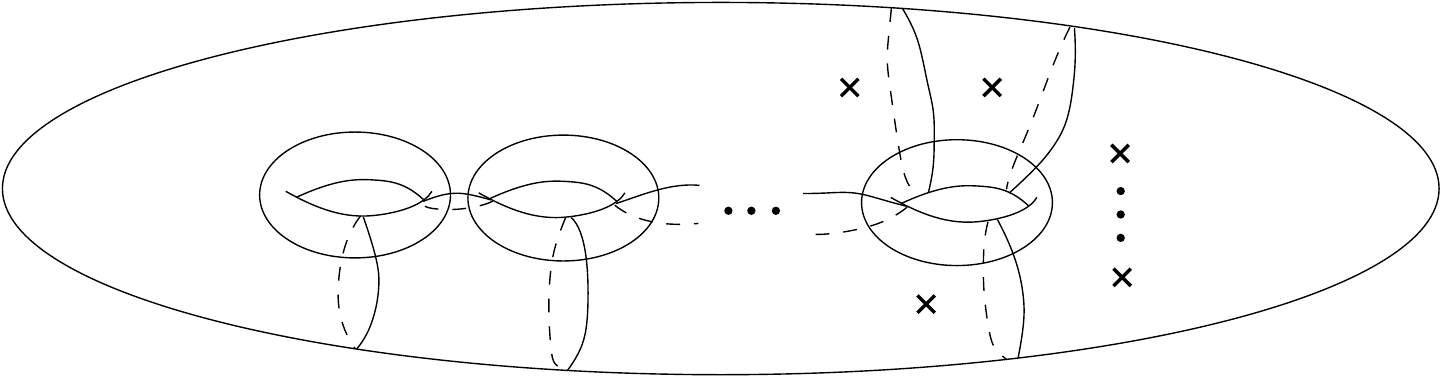}
\caption{Dehn twists about the curves shown generate the pure mapping class group $\pmcg(F)$. The crosses denote punctures.}
\label{fig:generators}
\end{figure}

Let us return to the case where $\mathcal Q$ is a collection of pairwise disjoint closed subsets of $\Ends(F)$.  
As $F$ is finite-type, each element $Q\in \mathcal Q$ is a finite collection of punctures.

\begin{Prop}\label{prop:sep2F}
If $|\mathcal Q| \geq 5$, then $\sep_2(F,\mathcal Q)$ is connected and infinite-diameter.  Furthermore, there exist infinitely many elements of $\mcg(F)$ acting on $\sep_2(F, \mathcal Q)$ with positive translation length.  
\end{Prop}
 
In the proof we will use the following connectivity criterion of Putman \cite[Lemma 2.1]{PutmanConnectivity}.

\begin{Lem}[Connectivity criterion]\label{lem:putman}
Let $G$ be a group acting on a simplicial complex $X$ and let $S$ be a generating set of $G$. Suppose there is a vertex $v$ of $X$ such that
\begin{enumerate}
\item for any vertex $w$, $G\cdot v$ intersects the connected component containing $w$, and
\item $v$ is connected to $s\cdot v$ for every $s\in S^{\pm 1}$.
\end{enumerate}
Then $X$ is connected.
\end{Lem}

\begin{proof}[Proof of Proposition \ref{prop:sep2F}]
 Let $\mathcal Q = \{Q_1, \ldots, Q_n\}$. 
Let $C=\{c_1, \ldots, c_n\}$ be a collection of pairwise disjoint separating curves such that one component of $F\smallsetminus c_i$, denoted $D_i$, is a disk containing $Q_i$.
Let $a \in \sep_2(F, \mathcal Q)$ be a curve that bounds a pair of pants with $c_1$ and $c_2$.  

Note that the assumption $|\mathcal Q| \geq 5$ is necessary: if the genus of $F$ is 0 and $|\mathcal Q| = 4$, then no two homotopically distinct curves $a$ and $b$ contained in $F \smallsetminus \bigcup_{i=1}^n D_i$ are disjoint, so we need to modify the adjacency relation (as suggested in the beginning of the section) to get a connected graph. If $|\mathcal Q|\leq 3$, the graph is empty.

If $F$ has positive genus, the Dehn twists about the simple closed curves shown in Figure \ref{fig:generators} generate the pure mapping class group $\pmcg(F)$ \cite[Corollary 4.16]{FarbPrimer}. If $F$ has genus-zero and $p$ punctures, we can obtain $F$ by deleting the vertices of a convex $p$-gon in $\br^2 \cup \{\infty\}$.  
The pure mapping class group is then generated by Dehn twists about the simple closed curves obtained as boundaries of  regular neighborhoods of the sides and the diagonals of the polygon \cite[Theorem 4.10]{MargalitGeometric}.
A standard classification of surfaces argument guarantees we can choose generators of $\pmcg(F)$ such that each Dehn twist in the generating set fixes all but at most one of the curves in $C$.
If $t$ is a Dehn twist in our generating set, then there exists $b\in\sep_2(F,Q)$ disjoint from $a$ and fixed by $t$, so $a$ and $t(a)$ have distance at most two.

Next, we claim that for every element $b\in \sep_2(F,\mathcal Q)$ there exists $g \in \pmcg(S)$ such that there exists a path between $b$ and $g(a)$ in $\sep_2(F,\mathcal Q)$.  
The $\pmcg(F)$-orbit of $b$ is determined by the genus of a component of $F\smallsetminus b$ and the partition $b$ gives of $\Ends(F)$.
Therefore, there exists $h \in \pmcg(F)$ such that $h(b)$ is contained in $\Sigma = F\smallsetminus \bigcup D_i$. 
By Theorem \ref{thm:sephyp}, $\sep_2(\Sigma, C)$ is connected.
As $\sep_2(\Sigma, C)$ is a subgraph of $\sep_2(F, \mathcal Q)$, there is a path connecting $a$ and $h(b)$.
By taking $g = h^{-1}$, we've proved the claim, which guarantees $\sep_2(F, \mathcal Q)$ is connected by Lemma \ref{lem:putman}.

As in the proof of Theorem \ref{thm:arcs}, we see that any pseudo-Anosov element of $\pmcg(F)$ acts on $\sep_2(F, \mathcal Q)$ with positive translation length.
\end{proof}

Given the statement of Theorem \ref{thm:sephyp}, it natural to ask:

\begin{Question}
Is $\sep_2(F, \mathcal Q)$ Gromov hyperbolic?  
\end{Question}

\section{Quasi-retracts}\label{sec:quasiretracts}
In \cite{AramayonaArc}, the authors prove a particular induced subgraph of the arc graph of a surface with punctures is hyperbolic, which generalizes a result of Bavard \cite{BavardHyperbolic}.  
Their proof is an elegant use of the fact that arc graphs of finite-type surfaces are uniformly hyperbolic \cite{HenselSlim}.  
The goal of this section is to abstract their argument for general use.

\subsection{Quasi-retracts}
Given a set $X$, let $\frak{P}(X)$ be the power set of $X$.  
For a geodesic metric space $(X, d_X)$ and a path-connected subspace $Y\subset X$, let $d_Y$ be the induced path metric on $Y$ (that is, the distance between two points in $Y$ is the infimum of the $d_X$-lengths of paths in $Y$ between the two points).

\begin{Def}
Let $(X,d)$ be a geodesic metric space and let $Y\subset X$ be path-connected.  
A function $r\co X \to \frak{P}(Y)\smallsetminus \{\emptyset\}$ is an \emph{$(A,B)$-quasi-retract} for $A,B\geq0$ if:
\begin{itemize}
\item
$\mathrm{diam}(r(x)) \leq A$ for all $x \in X$,
\item
$d_Y(r(x),r(x')) \leq B$ whenever $d_X(x,x') \leq 1$, and
\item
$r(y) = \{y\}$ for all $y \in Y$.
\end{itemize}
If $Y$ is a $(A,B)$-quasi-retract for some $A,B$, then we say $Y$ is a \emph{quasi-retract} of $X$. 
\end{Def} 
Similar notions (in particular that of a coarsely Lipschitz retract) appear in the literature (see for instance \cite{MjCannon-Thurston}); however, the above definition is better suited to our setting.

It follows easily from the definition that:

\begin{Lem}\label{lem:qi}
If $Y$ is an $(A,B)$-quasi-retract of $X$, then the inclusion $Y\hookrightarrow X$ is a $(2A+B,0)$-quasi-isometry.
\end{Lem}

As a consequence, we get the following.
\begin{Cor}\label{cor:infinite-diameter}
Let $Y$ be a quasi-retract of $X$. If $Y$ has infinite diameter, then so does $X$.
\end{Cor}

We now want to describe a criterion for a space to be Gromov hyperbolic. We will need the following definition.

\begin{Def}
Let $(X, d_X)$ be a geodesic metric space.  
A cover $\mathcal{Y}$ of $X$ is \emph{$\delta$-hyperbolic} if
\begin{enumerate}
\item
each $Y\in \mathcal{Y}$ is path-connected subspace of $X$,
\item
there exists $\delta >0$ such that each element $Y\in \mathcal{Y}$ is $\delta$-hyperbolic (with respect to the induced path metric), and
\item
for any geodesic triangle $T$ contained in $X$, there exists $Y\in \mathcal{Y}$ containing $T$.   
\end{enumerate}
\end{Def}
With this definition at hand, we are ready to state a criterion for hyperbolicity.

\begin{Lem}
\label{lem:hyperbolic}
If a geodesic metric space admits a  $\delta$-hyperbolic cover, then it is $\delta$-hyperbolic.
\end{Lem}

\begin{proof}
Let $X$ be a geodesic metric space with $\delta$-hyperbolic cover $\mathcal Y$.
Let $T$ be a geodesic triangle in $X$ and let $Y\in \mathcal{Y}$ such that $T\subset Y$.  
It follows that $T$ is geodesic in $Y$ implying it is $\delta$-slim.  
As the inclusion $Y \hookrightarrow X$ is distance non-increasing, $T$ is $\delta$-slim in $X$.  
It follows that $X$ is $\delta$-hyperbolic. 
\end{proof}

Motivated by the results in \cite{AramayonaGeometry}, we point out the following observation:

\begin{Lem}\label{lem:uniform}
Let $(X, d_X)$ be a geodesic metric space.  
Suppose there exists a cover $\mathcal{Y}$ of $X$ by quasi-retracts of $X$ satisfying: 
\begin{enumerate}
\item
the elements of $\mathcal{Y}$ are uniformly quasi-retracts of $X$, i.e.\ there exists constants $A,B\geq 0$ such that each $Y\in\mathcal{Y}$ is an $(A,B)$-quasi-retract of $X$, 
\item
for any geodesic triangle $T$ contained in $X$, there exists $Y\in \mathcal{Y}$ containing $T$. 
\end{enumerate}  
Then $X$ is $\delta$-hyperbolic if and only if each element of $\mathcal{Y}$ is $\delta$-hyperbolic.
\end{Lem}

\begin{proof}
If $X$ is hyperbolic, then each $Y\in \mathcal{Y}$ is quasi-isometrically embedded with the quasi-isometry constants being independent of $Y$.
Therefore, the elements of $Y\in \mathcal{Y}$ are uniformly hyperbolic.
On the other hand, if the elements of $\mathcal{Y}$ are $\delta$-hyperbolic, then $\mathcal{Y}$ is a $\delta$-hyperbolic cover of $X$.
It follows from Lemma \ref{lem:hyperbolic} that $X$ is $\delta$-hyperbolic.
\end{proof}

\subsection{Witnesses} 
An important notion for studying the geometry of subgraphs of the curve graph -- introduced in \cite{MasurDisk} -- is that of a \emph{witness}\footnote{A \emph{witness} was originally referred to as a \emph{hole}; this renaming was recommended by Saul Schleimer.}. 
In order to define witnesses, we first need the notion of subsurface projection.
\begin{Def}
Let $X$ be a non-annular\footnote{The definition can be modified to include the case of annuli; however, it is not needed here.} essential subsurface of $S$. The \emph{subsurface projection} $\pi_X:\cc(S)\to\mathfrak{P}(\cc(X))$ is given by:
\begin{itemize}
\item $\pi_X(a)=\{a\}$ if $a\subset X$,
\item $\pi_X(a)=\emptyset$ if $a\subset S\setminus X$,
\item $\pi_X(a)$ is the set of curves that can realized as a boundary of a regular neighborhood of $\alpha\cup\partial X$, where $\alpha$ is a representative of $a$ intersecting $\partial X$ minimally.
\end{itemize}
\end{Def}

\begin{Def}
Let $S$ be a surface and $\Gamma$ a subgraph of $\cc(S)$.
A \textit{witness} for $\Gamma$ is a subsurface $X \subset S$ such that every vertex in $\Gamma$ intersects $X$ nontrivially (that is, for every $a\in\Gamma$ the projection $\pi_X(a)$ is nonempty).
A witness $X$ is \emph{ideal} if the subsurface projection $\Gamma\to\mathfrak{P}(\cc(X))\smallsetminus\{\emptyset\}$ is a quasi-retract.
\end{Def}

\section{Theorem \ref{thm:main} for infinite-type surface}\label{sec:mainthminfinite}
We will now finish the proof of Theorem \ref{thm:main}.
For the entirety of this section, let $S$ be an infinite-type surface and let $\cp = \{P_1, \ldots, P_n\}$ be a finite collection of pairwise disjoint closed subsets of $\Ends(S)$ with $n \geq 5$.

\begin{Lem}
$\sep_2(S,\cp)$ is connected.
\end{Lem}

\begin{proof}
Let $a,b$ be non-adjacent elements of $\sep_2(S,\cp)$ and let $F\subset S$ be the surface that they fill.  
As $a,b$ are compact, $F$ is finite-type.  
By possibly enlarging $F$, we may assume that if $V$ is a component of $S\smallsetminus F$, then $V^* \subset \Ends(S)$ intersects at most one of the $P_i$.
To see the existence of this enlargement, observe that there exist pairwise disjoint clopen subsets $U_i$ of  $\Ends(S)$ such that $P_i \subset U_i$.
By intersecting $V^*$ with $U_i$, we see there is a compact surface $F' \supset F$ and a component $V'$ of $S\smallsetminus F'$ with $(V')^* = V^* \cap U_i$; by construction, $(V')^* \cap U_j = \emptyset$ for $j \neq i$.
Repeating this with every component of $S \smallsetminus F$ yields the desired enlargement.

For each $i\in \{1, \ldots, n\}$, define $\mathcal{V}_i$ to be the collection of components of $S\smallsetminus F$ satisfying $V\in \mathcal V_i$ if and only if $V^*\cap P_i \neq \emptyset$.  We then set
\[
Q_i = \bigcup_{V\in \mathcal V_i} \{c : c \text{ is a component of } \partial V\}
\] 
and $\mathcal Q = \{Q_1, \ldots, Q_n\}$.  
By Proposition \ref{prop:sep2F}, $\sep_2(F, \mathcal Q)$ is a connected subgraph of $\sep_2(S, \cp)$ containing both $a$ and $b$ yielding the result.
\end{proof}

\begin{figure}[t]
\includegraphics{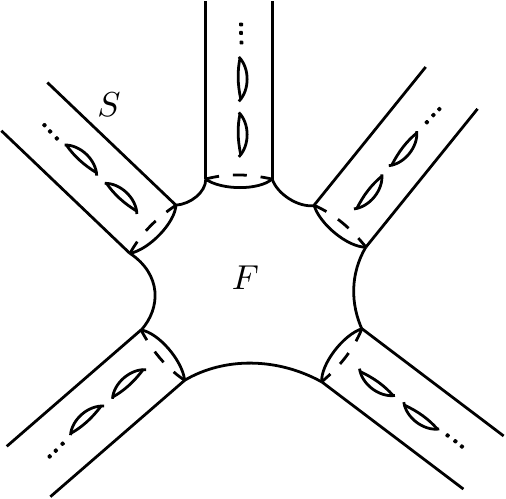}
\caption{The subsurface $F$ is an ideal witness for $\sep_2(S,\Ends(S))$.}
\label{fig:infinite}
\end{figure}

\begin{Lem}\label{lem:ideal2}
There exists a compact ideal witness for $\sep_2(S,\cp)$ with $|\cp|$ boundary components.
\end{Lem}

\begin{proof}
As $\Ends(S)$ is a normal topological space (i.e.\ it is $T_4$), we can find pairwise disjoint open sets $U_i$ containing $P_i$ for $1\leq i \leq n$.
Furthermore, as $\Ends(S)$ is totally disconnected, we may assume $\{U_i\}$ is a cover of $\Ends(S)$.
We can then choose separating (possibly peripheral) simple closed curves $c_i$ such there is a component $V_i$ of $S\smallsetminus c_i$ with $U_i = V_i^*$.  

Now, as $\{U_i\}$ is a finite cover of $\Ends(S)$, we must have that $F = S\smallsetminus \bigcup V_i$ is a compact subsurface.  
By possibly enlarging the $V_i$, we may suppose that $\genus(F) = 0$ (see Figure \ref{fig:infinite} for an example).

We claim that $F$ is an ideal witness for $\sep_2(S,\cp)$.
Let $c\in \sep_2(S,\cp)$ and suppose it is disjoint from $F$.  
It follows there is a component $V$ of $S\smallsetminus c$ containing $F$; hence, $V^*$ contains at least $|\cp|-1$ of the sets  in $\cp$, a contradiction.
We now have that $F$ is a witness; it is left to show that it is ideal.

Observe that $\cc(F)$ is a subgraph of $\sep_2(S,\cp)$.
It follows that the standard subsurface projection $\pi_F \co \sep_2(S,\cp) \to \mathfrak P (\cc(F))$ is a quasi-retract.
In particular, the subsurface projections of two disjoint curves have bounded intersection number and hence their distance is bounded (see \cite[Lemma 2.1]{MasurHyperbolicityI}).
This establishes that $F$ is an ideal witness.  
\end{proof}

\begin{Cor}
$\sep_2(S,P)$ is infinite diameter. Furthermore, there exist infinitely many elements of $\mcg(S)$ acting on $\sep_2(S, \cp)$ with positive translation length.
\end{Cor}

\begin{proof}
Let $F$ be the ideal witness constructed in Lemma \ref{lem:ideal2}. Then Lemma \ref{lem:qi} implies $\cc(F)$ is quasi-isometrically embedded in $\sep_2(S,\cp)$.
In particular, $\sep_2(S,\cp)$ is infinite-diameter and every pseudo-Anosov element in $\pmcg(F)\subset \pmcg(S)$ acts on $\sep_2(S,\cp)$ with positive translation length.
\end{proof}

Let us now assume that every set in $\cp$ is a singleton and let $P\subset \Ends(S)$ be the union of these singletons.

\begin{Lem}
$\sep_2(S,P)$ is $\delta$-hyperbolic, where $\delta$ can be chosen independently of $S$ and $P$.
\end{Lem}

\begin{proof}
Let $\mathcal{F}$ be the family of all finite-type surfaces $F \subset S$ such that the ends of each component of $S\smallsetminus F$ contains at most one point of $P$.
Observe that $\mathcal{F}$ is actually a cover of $S$.
For $F \in \mathcal{F}$, let $W_F$ be the components of $\partial F$ that partition $P$.
Note that $|W_F| = |P|$ and $\sep_2(F,W_F) \subset \sep_2(S,P)$.
Given a triangle $T \subset \sep_2(S,P)$ there exists $F\in \mathcal{F}$ such that $T \subset \sep_2(F,W_F)$.
This fact along with Theorem \ref{thm:sephyp} implies there exists $\delta$ such that $\{\sep_2(F,W_F)\}_{F\in \mathcal F}$ is a $\delta$-hyperbolic cover of $\sep_2(S, P)$.
The result follows by applying Lemma \ref{lem:hyperbolic}.
\end{proof}

\section{Nonseparating curve graphs}\label{sec:nonsep}
The results in this section are not new (see \cite{AramayonaGeometry, RasmussenUniform}); however, for completeness we provide proofs.  
Let $S$ be a finite-genus surface with $\genus(S) > 0$ and let $\nonsep(S)$ be the induced subgraph of $\cc(S)$ on the set of nonseparating curves. If $\genus(S)=1$ we modify the adjacency relation to intersecting once.

\begin{Lem}
\label{lem:nonsep}
If $F\subset S$ is a finite-type surface with $\genus(F) = \genus(S)$, then the inclusion $i\colon\!\nonsep(F)\hookrightarrow \nonsep(S)$ is an isometric embedding.
\end{Lem}

\begin{proof}
The components of $S\smallsetminus F$ are punctured disks.
By forgetting all but one puncture in each of these components, we obtain a natural simplicial map 
\[
p_F\colon\!\nonsep(S)\rightarrow \nonsep(F).
\]
It is clear that both $p_F$ and $i$ are 1-Lipschitz. 
Moreover, $p_F\circ i$ is the identity, so $i$ is an isometric embedding.
\end{proof}
\begin{Prop}\label{prop:hypiff}
Let $S$ be an infinite-type genus-$g$ surface with $0<g<\infty$, then $\nonsep(S)$ is hyperbolic if and only if all $\nonsep(S_{g,p})$ are $\delta$-hyperbolic with $\delta$ independent of $p$.
\end{Prop}

\begin{proof}
Recall that from the viewpoint of $\nonsep(S_{g,p})$ punctures and boundary components are interchangeable.
Let $\mathcal{Y}$ denote the collection of genus-$g$ finite-type subsurfaces of $S$.
Note that $\mathcal{Y}$ covers $\nonsep(S)$ and there is an element $\nonsep(F) \in \mathcal{Y}$ such that $F$ has $p$ boundary components for every $p\in \mathbb{Z}_+$. 

For each $\nonsep(F)\in \mathcal{Y}$, the map $p_F$ from Lemma \ref{lem:nonsep} is a (0,1)-quasi-retract.
Further, the finite collection of vertices of a geodesic triangle in $\nonsep(S)$ fill a finite-type surface, which can be enlarged to sit in $\mathcal{Y}$.
This shows that $\mathcal{Y}$ satisfies the conditions of Lemma \ref{lem:uniform} and the result follows.
\end{proof}

\begin{Cor}
If $\genus(S)$ is positive and finite, $\nonsep(S)$ is connected, infinite-diameter and there are infinitely many elements of $\mcg(S)$ acting on $\nonsep(S)$ with positive translation length. If furthermore $\genus(S)\geq 2$, $\nonsep(S)$ is $\delta$-hyperbolic, for $\delta$ independent of $S$.
\end{Cor}

\begin{proof}
If $S$ has finite-type, then $\nonsep(S)$ is connected (see \cite[Theorem 4.4]{FarbPrimer}) and infinite-diameter (as is every mapping class group invariant subgraph of $\cc(S)$).
Let $F\subset S$ be finite-type with $\genus(F) = \genus(S)$, then $\nonsep(F)$ is isometrically embedded in $\nonsep(S)$ by Lemma \ref{lem:nonsep}.
It follows that $\nonsep(S)$ is infinite-diameter and that any pseudo-Anosov supported in $F$ acts with positive translation length.
Further, as any two elements of $\nonsep(S)$ fill a finite-type surface, the connectedness of $\nonsep(S)$ follows from the finite-type case.
If $\genus(S)\geq 2$, using Proposition \ref{prop:hypiff} and Rasmussen's work \cite{RasmussenUniform} we get uniform hyperbolicity.
\end{proof}

\section{Low-index cases}\label{sec:oddities}
As mentioned in the introduction, having finite-invariance between two and three does not determine whether $S$ admits a connected graph consisting of curves on which the mapping class group acts with infinite-diameter orbits. In this section we present examples of surfaces that do admit such graphs and of surfaces that do not. In the case of $\fii=1$, we conjecture:

\begin{Conj}
If $S$ is infinite type and $\fii(S)=1$, then there does not exist a graph consisting of curves on $S$ with an infinite diameter orbit of $\mcg(S)$.
\end{Conj}

We call the surface with no punctures and exactly three ends accumulated by genus the \emph{tripod surface}. Since any permutation of the three ends can be realized by a mapping class, the only collection of subsets of the space of ends which is $\mcg(S)$-invariant (as in Definition \ref{def:fii}) is the set of the three singletons. As these are closed and proper subsets, we get that the tripod surface has $\fii=3$.

The \emph{spotted Loch Ness monster surface} will be the surface with exactly one end accumulated by genus and a sequence of isolated punctures converging to the nonplanar end.
Its space of ends is a sequence of isolated points (the punctures) converging to the unique end accumulated by genus. So the only closed  proper subset of the space of ends is the singleton formed by the end accumulated by genus, which is fixed by the mapping class group. This implies that the spotted Loch Ness monster has $\fii=1$.

\begin{figure}[h]
\includegraphics{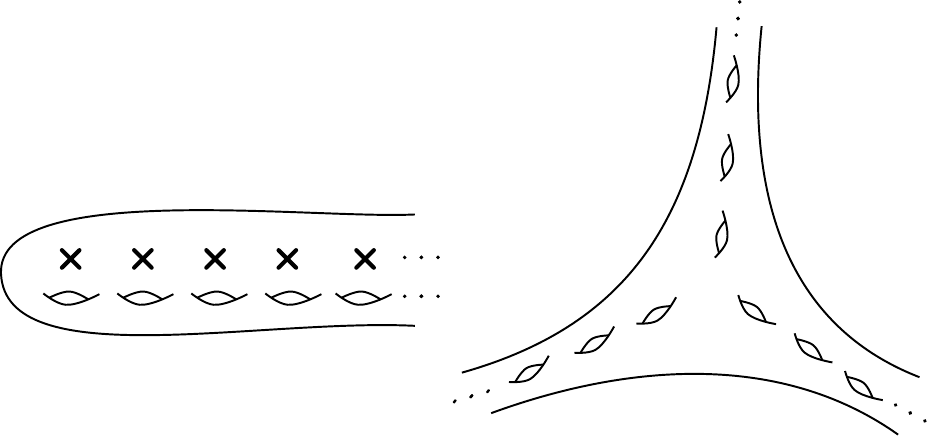}
\caption{The spotted Loch Ness monster surface (left) and the tripod surface (right).}
\end{figure}

\begin{Prop}
\label{prop:low}
Suppose $S$ is the spotted Loch Ness monster surface ($\fii=1$), the plane minus a Cantor set ($\fii=2$), or the tripod surface ($\fii=3$). Then any connected graph $\Gamma=\Gamma(S)$ consisting of curves with an action of $\mcg(S)$ has finite-diameter orbits.
\end{Prop}
To see that the plane minus a Cantor set has $\fii=2$, note that the Cantor set of punctures (denoted by $C$) and the isolated puncture (denoted by $\infty$) are closed in the space of ends. Furthermore, any two punctures in $C$ can be permuted by a mapping class, so no proper subset of it can be part of a $\mcg(S)$-invariant set. The isolated puncture instead is fixed by the mapping class group, so $\{C, \{\infty\}\}$ is $\mcg(S)$-invariant.
\begin{Rem}
The plane minus a Cantor set is the surface considered by Bavard in \cite{BavardHyperbolic}. In particular, it admits a connected induced subgraph of the arc graph on which the mapping class group acts with unbounded orbits.
This shows that graphs consisting of arcs might be better suited for some low-index cases.
\end{Rem}

\begin{proof}[Proof of Proposition \ref{prop:low}]
We rely again on Propositions \ref{prop:criterionfd} and \ref{prop:conditions}. Indeed, if $\Gamma$ contains either a nonseparating curve or a curve cutting off a finite-type surface, the conclusion follows from Proposition \ref{prop:conditions}. This covers all cases if $S$ is the spotted Loch Ness monster. If $S$ is the tripod surface, $\Gamma$ could contain a curve $c$ which is separating, but does not cut off a finite-type subsurface. In this case, we apply Proposition \ref{prop:criterionfd} with $$\cv=\{(a,b)\,|\,i(a,b)=0 \mbox{ and } S\setminus (a\cup b) \mbox{ has no finite-type component}\}.$$
Since $|\Ends(S)|=3$, there is a unique mapping class group orbit of pairs in $\cv$, and if a pair $a, b\in\mcg(S)\cdot c$ is not in $\cv$, there is a curve $d\in \mcg(S)\cdot c$ such that $(a,d)$ and $(b,d)$ belong to $\cv$ (because $a$ and $b$ fill some finite-type subsurface, so we can choose some curve $d$ far away that does not bound a finite-type subsurface with $a\cup b$).

If $S$ is the plane minus a Cantor set, we apply Proposition \ref{prop:criterionfd} with $$\cv=\{(a,b)\,|\, i(a,b)=0\}.$$ Note that there is a unique mapping class group orbit of curves on $S$, as the complement of each curve has two components, one being
homeomorphic to the open disk minus a Cantor set and the other the open annulus minus a Cantor set. Furthermore, there are two mapping class group orbits of pairs in $\cv$: for $(a,b)\in\cv$, either $S\setminus a\cup b$ is two disks minus a Cantor set and a once-punctured annulus or it is two disks minus a Cantor set and a once-punctured annulus minus a Cantor set. Since for any two intersecting curves $a,b$ there is a component of $S\setminus a\cup b$ containing a Cantor set, there is a curve $d$ disjoint from both $a$ and $b$.
\end{proof}

On the other hand, the surface with no punctures and exactly two ends accumulated by genus (\emph{Jacob's ladder surface}) has $\fii=2$ and we can show that it admits a graph consisting of curves with a good action of the mapping class group. More precisely, $\cg(S)$ will be the graph whose vertices correspond to curves separating the two non-planar ends, and where two curves are adjacent if they are disjoint and they cobound a genus one subsurface.

\begin{Lem}\label{lem:Gconnected}
$\cg(S)$ is connected and has infinite-diameter orbits.
\end{Lem}
\begin{proof}Consider $a,b\in\cg(S)$. If they are disjoint, then $S\setminus(a\cup b)$ has three components.
The component whose boundary consists of both  $a$ and $b$ is finite type, so it has some finite genus $k$.
It is then easy to construct a path between $a$ and $b$ of length $k$.

If instead they intersect, then we can choose a curve $c\in\cg(S)$ which is contained in the complement of the finite-type surface filled by $a$ and $b$. As this is disjoint from both $a$ and $b$, it is connected to both, so $a$ and $b$ are connected too.

To show that it has infinite-diameter orbits, note first that if two curves bound a genus $k$ subsurface, then they are at distance at most $k$. We want to show their distance is exactly $k$. Consider a path $a=c_0,c_1,\dots,c_m=b$. The union of the $c_i$ is contained in a finite-type subsurface of $S$, so there exist $d\in\cg(S)$ such that all $c_i$ are in the same connected component of $S\setminus d$. Let $g_i$ be the genus of the surface bounded by $c_i$ and $d$. At every step, $g_i$ can change by at most one and $|g_0-g_m|=k$. It follows that $m\geq k$. 

\begin{figure}[t]
\includegraphics{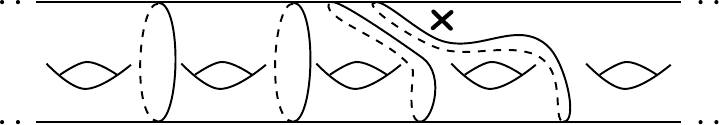}
\caption{The unbounded orbit of $\cg'(S)$ on the once-punctured Jacob's ladder surface $S$}\label{fig:1Jacob}
\end{figure}

The Jacob's ladder surface is a cyclic cover of the genus-2 surface.
Let $\tau$ denote deck transformation generating the deck group. 
Note that $\tau$ can be viewed as a translation (similar to Figure \ref{fig:not-hyperbolic}).
This shows that $\tau$ acts on $\cg(S)$ with translation length 1.
\end{proof}

\begin{Rem}
Let $S$ be obtained by puncturing the Jacob's ladder surface in a single point, then $\fii=3$.
If we define $\cg'(S)$ in the analogous way -- vertices correspond to separating curves separating the nonplanar ends and adjacency denotes cobounding a genus-one subsurface (either punctured or not) -- then Lemma \ref{lem:Gconnected} holds for $\cg'(S)$ with the same proof.
The infinite-diameter orbit is shown in Figure \ref{fig:1Jacob}.
\end{Rem}

\begin{Prop}
The graph $\cg(S)$ is not Gromov hyperbolic.
\end{Prop}

\begin{proof}
Fix $K >> 0$ and choose a compact genus-$K$ surface $\Sigma \subset S$ with two boundary components, denoted $a$ and $b$, each contained in $\cg(S)$ (see the surface cobounded by $a$ and $b$ in Figure \ref{fig:not-hyperbolic}). 
\begin{figure}[h]
\includegraphics[width=.6 \textwidth]{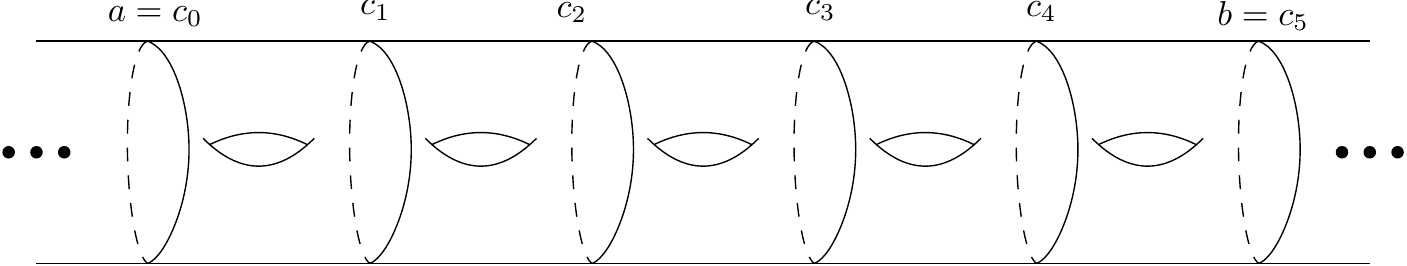}
\caption{A geodesic between $a$ and $b$ in $\cg(S)$.}
\label{fig:not-hyperbolic}
\end{figure}

Observe that for any $x,y\in \cg(S)$ contained in $\Sigma$
\[
d_{\cg(\Sigma)}(x,y) \geq d_{\cc(\Sigma)}(x,y),
\]
where $d_{\cc(\Sigma)}$ is the distance in $\cc(\Sigma)$ and $d_{\cg(\Sigma)}$ is the distance in $\cg(\Sigma)$.
Choose a geodesic $\gamma = (a=c_0, c_1, \ldots, c_K = b)$ in $\cg(S)$ (notice that the $c_i$ are contained in $\Sigma$).
Let $f\in  \mcg(\Sigma)$ be a pseudo-Anosov element.  Since $f$ acts loxodromically on $\cc(\Sigma)$, there exists $k \in \bz$ such that 
\[
d_{\cc(\Sigma)}(f^k(c_i), c_j) >  K
\]
for all $i,j \in \{1, \ldots, K-1\}$.
Set $c = c_{\lfloor K/2 \rfloor}$.  We claim that $d_{\cc(\Sigma)}(\gamma, c)$ is on the order of $K$.

For every $j \in \{1, \ldots, K-1\}$ there is a path connecting $c$ and $f^k(c_j)$ of length at most $K$ going through either $a$ or $b$ which travels along $\gamma$ and $f^k(\gamma)$.

Fix $j \in \{1, 2, \ldots, K-1\}$ and let $\eta = (c = e_0, e_1, \ldots, e_n = f^k(c_j))$ be a geodesic in $\cg(S)$.
Suppose every element of $\eta$ intersects $\Sigma$, then we can use subsurface projection to get a path in $\cc(\Sigma)$ whose length is on the order of $n$; in particular, $n$ must be on the order of $K$ by our assumptions.
Finally, suppose that $\eta$ contains a curve $e$ disjoint from $\Sigma$.
But, then either $d_{\cg(S)}(e, c) > d_{\cg(S)}(a,c)$ or $d_{\cg(S)}(e,c) > d_{\cg(S)}(b,c)$ implying $\eta$ has length greater than $\lfloor \frac K2 \rfloor$.  

Repeating this process for every such $K$, we see that geodesics between two points can be arbitrarily far apart.
It follows that $\cg(S)$ cannot be Gromov hyperbolic.
\end{proof}

\begin{Rem}
Brian Bowditch \cite{BowditchRank} proved, using connections with three-manifolds, that the graph $\cg(S)$ has geometric rank at least two, that is, there is a quasi-isometric embedding of $\bz^2$ into $\cg(S)$.
\end{Rem}

\appendix

\section{The ends}\label{appendix}
The underlying motivation for the constructions in this paper come from investigating the structure of the space of ends.
This structure does not directly show up in our proofs, but given its importance to our work, we explore the topology of the space of ends in this appendix.
We hope it will be helpful to others thinking about infinite-type surfaces. 

Note first that Richards \cite[Theorem 2]{RichardsClassification} proved that given any pair $(X,Y)$ of compact, separable, totally disconnected spaces satisfying $Y\subset X$ there exists a surface $S$ with $$(\Ends(S), \mathscr{AG}(S)) = (X,Y).$$

The discussion here can mostly be traced back to the Cantor-Bendixson derivative.
We refer the reader to \cite{KechrisClassical} for background on and proofs of the named theorems about Cantor sets that appear below.

We will discuss this derivative in the context of an infinite-type surface $S$.
It follows from Proposition \ref{prop:ends} and Brouwer's classification of the Cantor set that $\Ends(S)$ is homeomorphic to a closed subset of the Cantor set.
For the sequel, let $C$ denote the Cantor set and let $E\subset C$ be a closed subset.
Note that every closed subset of $C$ is totally disconnected, separable, and compact and can therefore be realized as $\Ends(S)$ for some $S$ (this is again \cite[Theorem 2]{RichardsClassification}). 

Given a topological space $X$, let $X'$ denote the set of limit points of $X$.
Using transfinite induction, the Cantor-Bendixson derivatives of $E$ are defined as follows:
\begin{itemize}
\item
$E^{(0)} = E$,
\item
$E^{(\al)} = \left(E^{(\al-1)}\right)'$ for any successor ordinal $\al$,
\item
$E^{(\al)} = \bigcap_{\be < \al} E^{(\be)}$ for any limit ordinal $\al$.  
\end{itemize}
Observe that each Cantor-Bendixson derivative of $E$ is invariant under $\mathrm{Homeo}(E)$.

Let us suppose that $E$ is countable.
Using the Cantor-Bendixson derivative, Mazurkiewicz-Sierpi\'nski \cite{MazurkiewiczContribution} gave a complete classification of countable compact Hausdorff spaces in terms of ordinals.
In particular, $E$ is homeomorphic to $\omega^\al n + 1$, where $\omega$ is the smallest infinite ordinal, $\al$ is a countable ordinal, and $n \in \bz_+$.
Here, we abused notation and conflated an ordinal $\al$ with the ordinal space consisting of all ordinals less than $\al$ endowed with the order topology.
The pair $(\al, n)$ is called to as the \emph{characteristic system} of $E$. 

If $E$ has characteristic system $(\al, n)$, then 
\begin{itemize}
\item
$|E^{(\beta)}| = \infty$ if $\be < \al$,
\item
$|E^{(\al)}| = n$, and
\item
$|E^{(\beta)}| = 0$ if $\be > \al$
\end{itemize}

\begin{Lem}\label{lem:countable}
Let $S$ be surface satisfying either  $\mathcal{AG}(S) = \emptyset$ or $\mathcal{AG}(S) = \Ends(S)$.
If $\Ends(S)$ is countable with characteristic system $(\al, n)$ and contains at least two points, then $\fii(S) = n$.
\end{Lem}

\begin{proof}
Let $E = \Ends(S)$ and let $E^{(\al)} = \{a_1, \ldots, a_n\}$.
Observe that $E^{(\al)}$ is an invariant set so $\fii(S) \geq n$.

Now let $\mathcal P = \{P_1, \ldots P_m\}$ be an invariant collection of pairwise disjoint proper closed subsets of $E$.
Let $\be < \al$ and let $x \in E^{(\be)}$.
If there exists $ i \in \{1, \ldots m \}$ such that  $x \in P_i$, then
\[
E^{(\be)} \subset \bigcup_{i=1}^m P_i
\]
as $E^{(\be)}$ is homogeneous.
Moreover, as $\be < \al$, we know $|E^{(\be)}|$ is infinite and therefore there exists $j_\be \in \{1, \ldots, m\}$ such that $|P_{j_\be} \cap E^{(\be)}|$ is infinite.
As $P_{j_\be}$ is compact and infinite, it must have a limit point; in particular, $P_{j_\be} \cap E^{(\be+1)}$ is nonempty.

For $k \in \mathbb N$, let $\be_k = \beta + k $.
Repeating the above argument, we see there exists $j_k \in \{1, \ldots, m \}$ such that $|P_{j_k}\cap E^{(\be_k)}|$ is infinite.
Another application of the pigeonhole principle implies that there exists $j \in \{1, \ldots, m\}$ such that $|P_j \cap E^{(\be_k)}|$ is infinite for infinitely many values of $k$.
As $P_j \cap E^{(\be_k)}$ is compact and $P_j \cap E^{(\be_k)} \supset P_j \cap E^{(\be_{k+1})}$ for all $k \in \mathbb N$, we see that
\[
\bigcap_{k\in\mathbb N} P_j \cap E^{(\be_k)} \neq \emptyset.
\]
In particular, if $\delta = \sup_k \{\beta_k\}$, then $P_j \cap E^{(\delta)}$ is nonempty.  

This inductive process guarantees that for each $ l \in \{1, \ldots n \}$ there exists $j_l \in \{1, \ldots, m\}$ such that $a_l \in P_{j_l}$.

Observe that if the collection $\mathcal P' = \mathcal P \smallsetminus \{P_{j_1}, \ldots, P_{j_n} \}$ is nonempty, then it must also be a finite invariant collection of closed proper subsets.
The same argument as above then implies there exists $P \in\mathcal P'$ containing an element of $E^{(\al)}$, but this is a contradiction since $P \cap P_{j_l} = \emptyset$ for all $l \in \{1, \ldots, n\}$.
We can conclude that $\mathcal P'$ is empty and hence $m \leq n$; therefore, $\fii(S) \leq n$.
\end{proof}

Let us now assume that $E$ is uncountable. 
It is an immediate corollary to the Cantor-Bendixson theorem and Brouwer's classification of the Cantor set that $E = C_0 \sqcup Q_0$, where $C_0 \cong C$ and $Q_0$ is homeomorphic to an open countable subset of $C$.
Also observe that $\partial Q_0 \subset C_0$ is a closed subset of $C$.
Reiterating the above discussion, we can build sets $C_\al$ and $Q_\al$ inductively for a cardinal $\al > 1$ as follows:
\begin{itemize}
\item
if $\al$ is a successor cardinal, set $\partial Q_{\al-1} = C_\al \sqcup Q_\al$.
\item
if $\al$ is a limit cardinal, set $\displaystyle{\bigcap_{\be < \al} C_\be = C_\al \sqcup Q_\al}$.
\end{itemize}
Observe that for each ordinal $\al$, the sets $C_\al, \partial Q_\al,$ and $\bar Q_\al$ are closed invariant subsets of $E$.
The point here is that $\mathrm{Homeo}(E)$-invariant sets can be buried deep in the Cantor set. 
This leads us to our last proposition characterizing $\fii = 0$ surfaces.  

\begin{Prop}\label{prop:appendix}
If $S$ is an infinite-type surface with $\fii(S) = 0$, then $S$ is one of:
\begin{enumerate}
\item
the Cantor tree surface,
\item
the blooming Cantor tree surface, or
\item
the Loch Ness monster surface.
\end{enumerate}
\end{Prop}

\begin{proof}
If $S$ has both planar ends and ends accumulated by genus, then $\mathcal{AG}(S)$ is a proper closed invariant subset and $\fii(S) > 0$.
We can therefore assume that either $\mathcal{AG}(S) = \emptyset$ or $\mathcal{AG}(S) = \Ends(S)$.

First suppose that $\Ends(S)$ is countable and let $(\al,n)$ be its characteristic system.
By Lemma \ref{lem:countable}, if $\Ends(S)$ has at least two points, then $\fii(S) = n \geq 1$.
It follows that $\Ends(S)$ must be a single point and therefore $S$ must be the one-ended infinite-genus surface and thus homeomorphic to the Loch Ness monster surface.

Now assume that $\Ends(S)$ is uncountable, so $\Ends(S) = C_0 \sqcup Q_0$, where $C_0$ is a Cantor set and $Q_0$ is countable.
If $Q_0$ is nonempty, then $C_0$ is a closed invariant proper subset and $\fii(S) > 0$.
It follows that $Q_0$ is empty and $\Ends(S)$ is a Cantor set.
By the definition of $\fii$, the genus of $S$ is either 0 or infinite.
As all the ends of $S$ are either planar or accumulated by genus, we can conclude that $S$ is either the Cantor tree surface or the blooming Cantor tree surface, respectively.
\end{proof}

\bibliographystyle{alpha}
\bibliography{big_graph}

\begin{thebibliography}{BKMM12}

\bibitem[AFP17]{AramayonaArc}
Javier Aramayona, Ariadna Fossas, and Hugo Parlier.
\newblock Arc and curve graphs for infinite-type surfaces.
\newblock {\em Proc. Amer. Math. Soc.}, 145(11):4995--5006, 2017.

\bibitem[AS60]{AhlforsRiemann}
Lars~V. Ahlfors and Leo Sario.
\newblock {\em Riemann surfaces}.
\newblock Princeton Mathematical Series, No. 26. Princeton University Press,
  Princeton, N.J., 1960.

\bibitem[AV16]{AramayonaGeometry}
Javier Aramayona and Ferr{\'a}n Valdez.
\newblock On the geometry of graphs associated to infinite-type surfaces.
\newblock {\em To appear in \emph{Math. Z.}}, 2016.

\bibitem[Bav16]{BavardHyperbolic}
Juliette Bavard.
\newblock Hyperbolicit{\'e} du graphe des rayons et quasi-morphismes sur un
  gros groupe modulaire.
\newblock {\em Geom. Topol.}, 20(1):491--535, 2016.

\bibitem[BBF10]{BBF10}
Mladen Bestvina, Ken Bromberg, and Koji Fujiwara.
\newblock Constructing group actions on quasi-trees and applications to mapping
  class groups.
\newblock {\em Publications math{\'e}matiques de l'IH{\'E}S}, pages 1--64,
  2010.

\bibitem[BH99]{BridsonMetric}
Martin~R. Bridson and Andr{{\'e}} Haefliger.
\newblock {\em Metric spaces of non-positive curvature}, volume 319 of {\em
  Grundlehren der Mathematischen Wissenschaften [Fundamental Principles of
  Mathematical Sciences]}.
\newblock Springer-Verlag, Berlin, 1999.

\bibitem[BKMM12]{BehrstockGeometry}
Jason Behrstock, Bruce Kleiner, Yair Minsky, and Lee Mosher.
\newblock Geometry and rigidity of mapping class groups.
\newblock {\em Geom. Topol.}, 16(2):781--888, 2012.

\bibitem[BM08]{BehrstockDimension}
Jason~A. Behrstock and Yair~N. Minsky.
\newblock Dimension and rank for mapping class groups.
\newblock {\em Ann. of Math. (2)}, 167(3):1055--1077, 2008.

\bibitem[Bow14]{BowditchUniform}
Brian~H. Bowditch.
\newblock Uniform hyperbolicity of the curve graphs.
\newblock {\em Pacific J. Math.}, 269(2):269--280, 2014.

\bibitem[Bow16]{BowditchRank}
Brian~H.\ Bowditch, 2016.
\newblock Personal communication.

\bibitem[Bow18]{BowditchLarge}
Brian~H. Bowditch.
\newblock Large-scale rigidity properties of the mapping class groups.
\newblock {\em Pacific J. Math.}, 293:1--73, 2018.

\bibitem[Cal04]{CalegariCircular}
Danny Calegari.
\newblock Circular groups, planar groups, and the {E}uler class.
\newblock In {\em Proceedings of the {C}asson {F}est}, volume~7 of {\em Geom.
  Topol. Monogr.}, pages 431--491. Geom. Topol. Publ., Coventry, 2004.

\bibitem[Cal09a]{CalegariBig}
Danny Calegari.
\newblock Big mapping class groups and dynamics.
\newblock Geometry and the imagination,
  https://lamington.wordpress.com/2009/06/22/big-mapping-class-groups-and-dynamics/,
  2009.

\bibitem[Cal09b]{CalegariBig2}
Danny Calegari.
\newblock Mapping class groups: the next generation.
\newblock Geometry and the imagination,
  https://lamington.wordpress.com/2014/10/24/mapping-class-groups-the-next-generation/,
  2009.

\bibitem[CC10]{CantwellEndperiodic}
John Cantwell and Lawrence Conlon.
\newblock Endperiodic automorphisms of surfaces and foliations.
\newblock {\em Preprint at https://arxiv.org/abs/1006.4525}, 2010.

\bibitem[FK04]{FunarUniversal}
L.~Funar and C.~Kapoudjian.
\newblock On a universal mapping class group of genus zero.
\newblock {\em Geom. Funct. Anal.}, 14(5):965--1012, 2004.

\bibitem[FK08]{FunarBraided}
Louis Funar and Christophe Kapoudjian.
\newblock The braided {P}tolemy-{T}hompson group is finitely presented.
\newblock {\em Geom. Topol.}, 12(1):475--530, 2008.

\bibitem[FK09]{FunarInfinite}
Louis Funar and Christophe Kapoudjian.
\newblock An infinite genus mapping class group and stable cohomology.
\newblock {\em Comm. Math. Phys.}, 287(3):784--804, 2009.

\bibitem[FM12]{FarbPrimer}
Benson Farb and Dan Margalit.
\newblock {\em A primer on mapping class groups}, volume~49 of {\em Princeton
  Mathematical Series}.
\newblock Princeton University Press, Princeton, NJ, 2012.

\bibitem[FP15]{FossasParlier}
Ariadna Fossas and Hugo Parlier.
\newblock Curve graphs on surfaces of infinite type.
\newblock {\em Ann. Acad. Sci. Fenn. Math., 40 (2), 2015. 793-801}, 2015.

\bibitem[Fun07]{FunarBraided2}
Louis Funar.
\newblock Braided {H}oughton groups as mapping class groups.
\newblock {\em An. \c Stiin\c t. Univ. Al. I. Cuza Ia\c si. Mat. (N.S.)},
  53(2):229--240, 2007.

\bibitem[Gil02]{GilmanDefinition}
Robert~H. Gilman.
\newblock On the definition of word hyperbolic groups.
\newblock {\em Math. Z.}, 242(3):529--541, 2002.

\bibitem[Har85]{HarerStability}
John~L. Harer.
\newblock Stability of the homology of the mapping class groups of orientable
  surfaces.
\newblock {\em Ann. of Math. (2)}, 121(2):215--249, 1985.

\bibitem[Har86]{HarerVirtual}
John~L. Harer.
\newblock The virtual cohomological dimension of the mapping class group of an
  orientable surface.
\newblock {\em Invent. Math.}, 84(1):157--176, 1986.

\bibitem[HPW15]{HenselSlim}
Sebastian Hensel, Piotr Przytycki, and Richard C.~H. Webb.
\newblock 1-slim triangles and uniform hyperbolicity for arc graphs and curve
  graphs.
\newblock {\em J. Eur. Math. Soc. (JEMS)}, 17(4):755--762, 2015.

\bibitem[Kec95]{KechrisClassical}
Alexander~S. Kechris.
\newblock {\em Classical descriptive set theory}, volume 156 of {\em Graduate
  Texts in Mathematics}.
\newblock Springer-Verlag, New York, 1995.

\bibitem[Mj14]{MjCannon-Thurston}
Mahan Mj.
\newblock Cannon-{T}hurston maps for surface groups.
\newblock {\em Ann. of Math. (2)}, 179(1):1--80, 2014.

\bibitem[MM99]{MasurHyperbolicityI}
Howard~A. Masur and Yair~N. Minsky.
\newblock Geometry of the complex of curves. {I}. {H}yperbolicity.
\newblock {\em Invent. Math.}, 138(1):103--149, 1999.

\bibitem[MM09]{MargalitGeometric}
Dan Margalit and Jon McCammond.
\newblock Geometric presentations for the pure braid group.
\newblock {\em J. Knot Theory Ramifications}, 18(1):1--20, 2009.

\bibitem[MR16]{MannLarge}
Kathryn Mann and Christian Rosendal.
\newblock Large scale geometry of homeomorphism groups.
\newblock {\em \emph{To appear in} Ergodic Theory Dyn. Syst.}, 2016.

\bibitem[MS20]{MazurkiewiczContribution}
Stefan Mazurkiewicz and Wac{\l}aw Sierpi{\'n}ski.
\newblock Contribution {\`a} la topologie des ensembles d{\'e}nombrables.
\newblock {\em Fundamenta Mathematicae}, 1(1):17--27, 1920.

\bibitem[MS13]{MasurDisk}
Howard Masur and Saul Schleimer.
\newblock The geometry of the disk complex.
\newblock {\em J. Amer. Math. Soc.}, 26(1):1--62, 2013.

\bibitem[Put08]{PutmanConnectivity}
Andrew Putman.
\newblock A note on the connectivity of certain complexes associated to
  surfaces.
\newblock {\em Enseign. Math. (2)}, 54(3-4):287--301, 2008.

\bibitem[{Ras}17]{RasmussenUniform}
A.~J. {Rasmussen}.
\newblock {Uniform hyperbolicity of the graphs of nonseparating curves via
  bicorn curves}.
\newblock {\em ArXiv e-prints}, July 2017.

\bibitem[Ric63]{RichardsClassification}
Ian Richards.
\newblock On the classification of noncompact surfaces.
\newblock {\em Trans. Amer. Math. Soc.}, 106:259--269, 1963.

\bibitem[Ros16]{RosendalCoarse}
Christian Rosendal.
\newblock Coarse geometry of topological groups.
\newblock {\em Preprint at
  \url{http://homepages.math.uic.edu/~rosendal/PapersWebsite/Coarse-Geometry-Book17.pdf}},
  2016.

\end{thebibliography}

\end{document}